\documentclass{amsart}
\usepackage{amsfonts,amssymb,amsmath,amsthm}
\usepackage{url}
\usepackage{enumerate}

\newtheorem{thrm}{Theorem}[section]
\newtheorem{lem}[thrm]{Lemma}
\newtheorem{prop}[thrm]{Proposition}
\newtheorem{cor}[thrm]{Corollary}
\theoremstyle{definition}
\newtheorem{definition}[thrm]{Definition}

\numberwithin{equation}{section}

\author{Bayram \d{S}ahin}
\address{
Department of Mathematics\\
Inonu University\\
Malatya, Turkey.} \email{bsahin@inonu.edu.tr}
\thanks{}

\begin{document}

\title[Semi-invariant Submersions]{Semi-invariant submersions from almost Hermitian
manifolds}

\begin{abstract}We introduce semi-invariant Riemannian submersions from almost
Hermitian manifolds onto Riemannian manifolds. We give  examples,
investigate the geometry of foliations which are arisen from the
definition of a Riemannian submersion and find necessary-sufficient
conditions for total mani-
fold to be a locally product Riemannian
manifold. We also find necessary and sufficient conditions for a
semi-invariant submersion to be totally geodesic. Moreover, we
obtain  a classification for semi-invariant submersions with totally
umbilical fibers and show that such submersions put some
restrictions on total manifolds.
\end{abstract}
\maketitle

\section{Introduction}\label{sect1}

A Riemannian submersion is a smooth submersion $F: M_1
\longrightarrow M_2$ between two Riemannian manifolds $(M_1,g_1)$
and $(M_2,g_2)$ with the property that at any point $p \in M_1$,
$$g_{1p}(x,y)=g_{2F(p)}(F_*(x),F_*(y))$$
for any $x,y$ in the tangent space $T_pM_1$ to $M_1$ at $p \in M_1$,
that are perpendicular to the kernel of $F_*$.\\

  Riemannian
submersions between Riemannian manifolds were studied by O'Neill
\cite{O'Neill} and Gray \cite{Gray}. Later such submersions have
been studied widely in differential geometry. Riemannian submersions
between Riemannian manifolds equipped with an additional structure
of almost complex type was firstly studied by Watson in
\cite{Watson}. Watson defined almost Hermitian submersions between
almost Hermitian manifolds and he showed that the base manifold and
each fiber have the same kind of structure as the total space, in
most cases. More precisely, let $M_1$ be a complex $m-$dimensional
almost Hermitian manifold with Hermitian metric $g_1$ and almost
complex structure $J_1$ and $M_2$ be a complex $n-$dimensional
almost Hermitian ma- nifold with Hermitian metric $g_2$ and almost
complex structure $J_2$. A Riemannian submersion
$F:M_1\longrightarrow M_2$ is called an almost Hermitian submersion
if $F$ is an almost complex mapping, i.e., $F_*J_1=J_2F_*$. The main
result of this notion is that the vertical and horizontal
distributions are $J_1-$ invariant. On the other hand, Escobales
\cite{Escobales} studied Riemannian submersions from complex
projective space onto a Riemannian manifold under the assumption
that the fibers are connected, comp-
lex, totally geodesic
submanifolds. In fact, this assumption also implies that the
vertical distribution is invariant with respect to the almost
complex structure.  We note that almost Hermitian submersions have
been extended to the almost contact manifolds \cite{Domingo},
locally conformal K\"{a}hler manifolds \cite{lck} and quaternion
K\"{a}hler manifolds \cite{Ianus}.\\

 All these submersions mentioned
above have one common property. In these submersions vertical and
horizontal distributions are invariant. Therefore, recently we have
introduced the notion of anti-invariant Riemannian submersions which
are Riemannian submersions from almost Hermitian manifolds such that
their vertical distribution is anti-invariant under the almost
complex structure of
the total ma-
nifold, \cite{Sahin}.\\

In this paper, we introduce semi-invariant Riemannian submersions as
a ge- neralization of anti-invariant Riemannian submersions and
almost Hermitian submersions when the base manifold is an almost
Hermitian manifold. We show that such submersions are useful to
investigate
the geometry of the total manifold of the submersion.\\

The paper is organized as follows. In section 2, we give brief
information about almost Hermitian manifolds, Riemannian
submersions and distributions which are defined by the Riemannian
submersion. In section 3, we define semi-invariant Riemannian
submersion, give examples and investigate the geometry of its
leaves. Then we use these results to obtain decomposition theorems
for the total manifold. We also find necessary and sufficient
conditions for semi-invariant submersions to be totally geodesic.
In section 4, we first show that the notion of semi-invariant
submersions puts some restrictions on the sectional curvature of
the total manifold when it is a complex space form. Then we obtain
a classification theorem of semi-invariant submersions
with totally umbilical fibers.\\

\section{Preliminaries}\label{pre}
In this section, we define almost Hermitian manifolds, recall the
notion of Riemannian submersions between Riemannian manifolds and
give a brief review
of basic facts of Riemannian submersions.\\

Let ($\bar{M}, g$) be an almost Hermitian manifold. This means
\cite{Yano-Kon} that $\bar{M}$ admits a tensor field $J$ of type (1,
1) on $\bar{M}$ such that, $\forall X, Y \in \Gamma(T\bar{M})$, we
have
\begin{equation}
J^2=-I, \quad g(X, Y)=g(JX, JY) \label{eq:2.1}.
\end{equation}
 An almost Hermitian manifold $\bar{M}$ is called  K\"{a}hler manifold if
\begin{equation}
(\bar{\nabla}_XJ)Y=0,\forall X,Y \in \Gamma(T\bar{M}),
\label{eq:2.2}
\end{equation}
where $\bar{\nabla}$ is the Levi-Civita connection on $\bar{M}.
$\\

Let $(M_1^m,g_1)$ and $(M_2^n,g_2)$ be Riemannian manifolds, where
$dim(M_1)=m$, $dim(M_2)=n$ and $m>n$. A Riemannian submersion
$F:M_1\longrightarrow M_2$ is a map from $M_1$ onto $M_2$
satisfying the following axioms:
\begin{itemize}
    \item [(S1)] $F$ has maximal rank.
    \item [(S2)] The differential $F_*$ preserves the lenghts of
    horizontal vectors.
\end{itemize}
For each $q\in M_2$, $F^{-1}(q)$ is an $(m-n)$ dimensional
submanifold of $M_1$. The submanifolds $F^{-1}(q)$, $q\in M_2$, are
called fibers. A vector field on $M_1$ is called vertical if it is
always tangent to fibers. A vector field on $M_1$ is called
horizontal if it is always orthogonal to fibers. A vector field $X$
on $M_1$ is called basic if $X$ is horizontal and $F-$ related to a
vector field $X_*$ on $M_2$, i.e., $F_*X_p=X_{*F(p)}$ for all $p \in
M_1$. Note that we denote the projection morphisms on the
distributions $ker F_*$ and $(kerf F_*)^\perp$ by $\mathcal{V}$ and
$\mathcal{H}$, respectively.\\

We recall the following lemma from O'Neill \cite{O'Neill}.
\begin{lem}\label{Oneil}Let $F:M_1 \longrightarrow
M_2$ be a Riemannian submersion between Riemannian manifolds and $X,
Y$ be basic vector fields of $M_1$. Then we have
\begin{itemize}
    \item [(a)] $g_1(X, Y)=g_2(X_*,Y_*)\circ F$,
    \item [(b)] the horizontal part $[X,Y]^{\mathcal{H}}$ of
    $[X,Y]$ is a basic vector field and corresponds to
    $[X_*,Y_*]$, i.e., $F_*([X,Y]^{\mathcal{H}})=[X_*,Y_*]$.
    \item [(c)] $[V,X]$ is vertical for any  vector field $V$ of
    $ker F_*$.
    \item [(d)] $(\nabla^{^1}_XY)^{\mathcal{H}}$ is the basic vector field  corresponding to $\nabla^{^2}_{X_*}Y_*$,
\end{itemize}
where $\nabla^{^1}$ and $\nabla^{^2}$ are the Levi-Civita
connections of $g_1$ and $g_2$, respectively.
\end{lem}

 The geometry of
Riemannian submersions is characterized by O'Neill's tensors
$\mathcal{T}$ and $\mathcal{A}$ defined for vector fields $E, F$
on $M_1$ by
\begin{equation}
\mathcal{A}_E
F=\mathcal{H}\nabla^{^1}_{\mathcal{H}E}\mathcal{V}F+\mathcal{V}\nabla^{^1}_{\mathcal{H}E}\mathcal{H}F
\label{eq:2.3}
\end{equation}
\begin{equation}
\mathcal{T}_E
F=\mathcal{H}\nabla^{^1}_{\mathcal{V}E}\mathcal{V}F+\mathcal{V}\nabla^{^1}_{\mathcal{V}E}\mathcal{H}F.
\label{eq:2.4}
\end{equation}
It is easy to see that a Riemannian submersion
$F:M_1\longrightarrow M_2$ has totally geodesic fibers if and only
if $\mathcal{T}$ vanishes identically. For any $E \in
\Gamma(TM_1)$, $\mathcal{T}_E$ and $\mathcal{A}_E$ are
skew-symmetric operators on $(\Gamma(TM_1),g)$ reversing the
horizontal and the vertical distributions. It is also easy to see
that $\mathcal{T}$ is vertical,
$\mathcal{T}_E=\mathcal{T}_{\mathcal{V}E}$ and $\mathcal{A}_E$ is
horizontal, $\mathcal{A}_E=\mathcal{A}_{\mathcal{H}E}$. We note
that the tensor fields $\mathcal{T}$ and $\mathcal{A}$ satisfy
\begin{eqnarray}
\mathcal{T}_UW&=&\mathcal{T}_WU,\, \forall U,W \in \Gamma(ker F_*)\label{eq:2.5}\\
\mathcal{A}_XY&=&-\mathcal{A}_YX=\frac{1}{2}\mathcal{V}[X,Y],\,
\forall X,Y\in \Gamma((ker F_*)^{\perp}).\label{eq:2.6}
\end{eqnarray}
On the other hand, from (\ref{eq:2.3}) and (\ref{eq:2.4}) we have
\begin{eqnarray}
\nabla^{^1}_VW&=&\mathcal{T}_VW+\hat{\nabla}_V W \label{eq:2.7}\\
\nabla^{^1}_VX&=&\mathcal{H}\nabla^{^1}_VX+\mathcal{T}_VX \label{eq:2.8}\\
\nabla^{^1}_XV&=&\mathcal{A}_XV+\mathcal{V}\nabla^{^1}_XV \label{eq:2.9}\\
\nabla^{^1}_XY&=&\mathcal{H}\nabla^{^1}_XY+\mathcal{A}_XY
\label{eq:2.10}
\end{eqnarray}
for $X,Y \in \Gamma((ker F_*)^{\perp})$ and $V,W \in \Gamma(ker
F_*)$, where $\hat{\nabla}_VW=\mathcal{V}\nabla^{^1}_VW$. If $X$
is
basic, then $\mathcal{H}\nabla^{^1}_VX=\mathcal{A}_XV$.\\

Finally, we recall the notion of the second fundamental form of a
map between Riemannian manifolds. Let $(M_1, g_1)$ and $(M_2, g_2)$
be Riemannian manifolds and suppose that $\varphi:
M_1\longrightarrow M_2$ is a smooth map between them. Then the
differential $\varphi_{*}$ of $\varphi$ can be viewed as a section
of the bundle $Hom(TM_1,\, \varphi^{-1}TM_2)\, \longrightarrow \,
M_1$, where $\varphi^{-1}TM_2$ is the pullback bundle which has
fibers $(\varphi^{-1}TM_2)_p=T_{\varphi(p)} M_2, p \in M_1.$
$Hom(TM_1, \varphi^{-1}TM_2)$ has a connection $\nabla$ induced from
the Levi-Civita connection $\nabla^{^1}$ and the pullback
connection. Then the second fundamental form of $\varphi$ is given
by
\begin{equation}
(\nabla \varphi_{*})(X, Y)=\nabla^{\varphi}_X
\varphi_{*}(Y)-\varphi_{*}(\nabla^{^1}_X Y) \label{eq:2.11}
\end{equation}
for $X, Y \in \Gamma(TM_1)$, where $\nabla^{\varphi}$ is the
pullback connection. It is known that the second fundamental form is
symmetric.\\

\section{Semi-invariant Riemannian submersions}\label{sct3}

In this section, we define semi-invariant Riemannian submersions
from an almost Hermitian manifold onto a Riemannian manifold,
investigate the integrability of distributions and obtain a
necessary and sufficient condition for such submersions to be
totally geodesic map. We also obtain two decomposition theorems for
the total manifolds of such submersions.

\begin{definition}Let $M_1$ be a complex $m-$
dimensional almost Hermitian manifold with Hermitian metric $g_1$
and almost complex structure $J$ and $M_2$ be a Riemannian manifold
with Riemannian metric $g_2$. A Riemannian submersion $F:M_1
\longrightarrow M_2$ is called semi-invariant Riemannian submersion
if there is a distribution $\mathcal{D}_1 \subseteq ker F_*$  such that
\begin{equation}
ker F_*=\mathcal{D}_1 \oplus \mathcal{D}_2 \label{eq:3.1}
\end{equation} and
\begin{equation}
J(\mathcal{D}_1)=\mathcal{D}_1,\,J(\mathcal{D}_2)\subseteq (ker
F_*)^\perp, \label{eq:3.2}
\end{equation}
where $\mathcal{D}_2$ is orthogonal complementary to $\mathcal{D}_1$ in $ker F_*$.
\end{definition}

We note that it is known that the distribution $ker F_*$ is
integrable. Hence, above definition implies that the integral
manifold (fiber) $F^{-1}(q)$, $q\in M_2$, of $ker F_*$ is a
CR-submanifold of $M_1$. For CR-submanifolds, see: \cite{Bejancu},
\cite{Chen} and \cite{Yano-Kon2}. We now give some
examples of semi-invariant Riemannian submersions.\\

\noindent{\bf Example~1.~}Every anti-invariant Riemannian submersion
from an almost Hermitian manifold onto a Riemannian manifold is a
semi-invariant Riemannian submersion with
$\mathcal{D}_1=\{0\}.$\\

\noindent{\bf Example 2.~}Every Hermitian submersion from an
almost Hermitian manifold onto an almost Hermitian manifold is a
semi-invariant submersion with $\mathcal{D}_2=\{0\}$.\\

\noindent{\bf Example~3.~} Let $F$ be a submersion defined by
$$
\begin{array}{cccc}
  F: & R^6             & \longrightarrow & R^3\\
     & (x_1,x_2,x_3,x_4,x_5,x_6) &             & (\frac{x_1 + x_2}{\sqrt{2}},
     \frac{x_3+x_5}{\sqrt{2}}, \frac{x_4+x_6}{\sqrt{2}}).
\end{array}
$$
Then it follows that
$$ker F_*=span\{V_1=-\partial x_1+\partial x_2, V_2=-\partial x_3+\partial
x_5, V_3=-\partial x_4+\partial x_6\}$$ and
$$(ker F_*)^\perp=span\{X_1=\partial x_1+\partial x_2, X_2=\partial x_3+\partial
x_5, X_3=\partial x_4+\partial x_6\}.$$ Hence we have $JV_2=V_3$
and $JV_1=-X_1$. Thus it follows that
$\mathcal{D}_1=span\{V_2,V_3\}$ and $\mathcal{D}_2=span\{V_1\}$.
Moreover one can see that $\mu=span \{X_2,X_3\}$. By direct
computations, we also have
$$
g_{R^6}(JV_1,JV_1)=g_{R^3}(F_*(JV_1),F_*(JV_1)),\,g_{R^6}(X_2,X_2)=g_{R^3}(F_*(X_2),F_*(X_2))$$
and $$g_{R^6}(X_3,X_3)=g_{R^3}(F_*(X_3),F_*(X_3)),
$$
which show that $F$ is a Riemannian submersion. Thus $F$ is a
semi-invariant Riemannian submersion.\\

We now investigate the integrability of the distributions
$\mathcal{D}_1$ and $\mathcal{D}_2$. Since fibers of semi-invariant
submersions from K\"{a}hler manifolds are CR-submanifolds and
$\mathcal{T}$ is the second fundamental form of the fibers, the
following results can be deduced from Theorem 1.1 of
\cite[p.39]{Bejancu}.

\begin{lem}\label{Lemma 3.2}Let $F$ be a semi-invariant Riemannian submersion from a
K\"{a}hler manifold $(M_1,g_1,J_1)$ onto a Riemannian manifold
$(M_2,g_2 )$. Then
\begin{enumerate}[(i)]
    \item the distribution $\mathcal{D}_2$ is always integrable.
    \item The distribution $\mathcal{D}_1$ is integrable if and
    only if
    $$g_1(T_XJY-T_YJX,JZ)=0$$
    for $X, Y\in \Gamma(\mathcal{D}_1)$ and $Z \in \Gamma(\mathcal{D}_2)$.
\end{enumerate}
\end{lem}

 Let $F$ be a
semi-invariant Riemannian submersion from a K\"{a}hler manifold
$(M_1,g_1,J)$ onto a Riemannian manifold $(M_2,g_2)$. We denote
the complementary distribution to $J\mathcal{D}_2$ in $(ker
F_*)^\perp$ by $\mu$. Then for $V \in \Gamma(ker F_*)$, we write
\begin{equation}
JV=\phi V+\omega V, \label{eq:3.3}
\end{equation}
where $\phi V \in \Gamma(\mathcal{D}_1)$ and $ \omega V \in
\Gamma(J\mathcal{D}_2)$. Also for $X \in \Gamma((ker F_*)^\perp)$,
we have
\begin{equation}
JX=\mathcal{B}X+\mathcal{C}X, \label{eq:3.4}
\end{equation}
where $\mathcal{B}X \in \Gamma(\mathcal{D}_2)$ and $\mathcal{C}X \in
\Gamma(\mu)$. Then, by using (\ref{eq:3.3}), (\ref{eq:3.4}),
(\ref{eq:2.7}) and (\ref{eq:2.8}) we get
\begin{eqnarray}
(\nabla_V \phi )W&=&\mathcal{B}\mathcal{T}_VW-\mathcal{T}_V \omega
W
\label{eq:3.5}\\
(\nabla_V
 \omega )W&=&\mathcal{C}\mathcal{T}_VW-\mathcal{T}_V\phi W,
\label{eq:3.6}
\end{eqnarray}
for $V, W \in \Gamma(ker F_*)$, where
$$ (\nabla_V
\phi )W=\hat{\nabla}_V\phi W-\phi \hat{\nabla}_VW$$ and
$$(\nabla_V
 \omega )W=\mathcal{H}\nabla^{^1}_V \omega W- \omega
 \hat{\nabla}_VW.$$

The proof of the following proposition can be deduced from Theorem
5.1 of \cite[p.63]{Bejancu}.

\begin{prop}\label{Proposition 3.3}Let $F$ be a
semi-invariant Riemannian submersion from a K\"{a}hler manifold
$(M_1,g_1,J)$ onto a Riemannian manifold $(M_2,g_2)$. Then the
fibers of $F$ are locally product Riemannian manifolds if and only
if $(\nabla_V\phi )W=0$ for $V, W \in \Gamma(ker F_*)$.
\end{prop}

We now obtain necessary and sufficient conditions for a
semi-invariant submersion to be totally geodesic. We recall that a
differentiable map $F$ between Riemannian manifolds $(M_1,g_1)$ and
$(M_2,g_2)$ is called a totally geodesic map if $(\nabla
F_*)(X,Y)=0$ for all $X, Y \in \Gamma(TM_1)$.

\begin{thrm}\label{Theorem 3.4}Let $F$ be a semi-invariant submersion from a K\"{a}hler
manifold $(M_1,g_1,J)$  onto a Riemannian manifold $(M_2,g_2)$. Then
$F$ is a totally geodesic map if and only if
\begin{enumerate}[(a)]
    \item $\hat{\nabla}_X\phi Y+\mathcal{T}_X\omega Y$ and
$\hat{\nabla}_X\mathcal{B}Z+\mathcal{T}_X\mathcal{C}Z$ belong to
$\mathcal{D}_1$.
    \item $\mathcal{H}\nabla^{^1}_X\omega Y+T_X \phi Y$ and
    $\mathcal{T}_X\mathcal{B}Z+\mathcal{H}\nabla^{^1}_X\mathcal{C}Z$ belong to $J\mathcal{D}_2$
\end{enumerate} for $X,Y \in \Gamma(ker F_*)$ and $Z \in \Gamma((ker
  F_*)^\perp)$.
  \end{thrm}

\begin{proof} First of all, since $F$ is a Riemannian
submersion we have
\begin{equation}
(\nabla F_*)(Z_1,Z_2)=0, \forall Z_1,Z_2 \in \Gamma((ker
F_*)^\perp).\label{eq:3.8}
\end{equation}
 For $X, Y \in \Gamma(ker F_*)$, we get $(\nabla
 F_*)(X,Y)=-F_*(\nabla^{^1}_XY)$. Then from (\ref{eq:2.2}) we get
 $(\nabla F_*)(X,Y)=F_*(J\nabla^{^1}_XJY)$. Using (\ref{eq:3.3}) we have
 $(\nabla F_*)(X,Y)=F_*(J\nabla^{^1}_X\phi Y+J\nabla^{^1}_X \omega Y)$. Then
 from (\ref{eq:2.7}) and (\ref{eq:2.8}) we arrive at
$$(\nabla F_*)(X,Y)=F_*(J(\hat{\nabla}_X\phi Y+\mathcal{T}_X \phi Y+\mathcal{H}\nabla^{^1}_X \omega Y+\mathcal{T}_X \omega Y)).$$
Using (\ref{eq:3.3}) and (\ref{eq:3.4}) in above equation we obtain
\begin{eqnarray}
(\nabla F_*)(X,Y)&=&F_*(\phi \hat{\nabla}_X\phi
Y+\omega\hat{\nabla}_X\phi Y+\mathcal{B}\mathcal{T}_X \phi
Y\nonumber\\
&+&\mathcal{C}\mathcal{T}_X \phi
Y+\mathcal{B}\mathcal{H}\nabla^{^1}_X
\omega Y+\mathcal{C}\mathcal{H}\nabla^{^1}_X \omega Y\nonumber\\
&+&\phi \mathcal{T}_X \omega Y+\omega \mathcal{T}_X \omega
Y).\nonumber
\end{eqnarray}
Since $\phi \hat{\nabla}_X\phi Y+\mathcal{B}\mathcal{T}_X \phi
Y+\phi \mathcal{T}_X \omega Y+\mathcal{B}\mathcal{H}\nabla^{^1}_X
\omega Y \in \Gamma(ker F_*)$, we derive
\begin{eqnarray}
(\nabla F_*)(X,Y)&=&F_*(\omega\hat{\nabla}_X\phi
Y+\mathcal{C}\mathcal{T}_X \phi Y\nonumber\\
&+&\mathcal{C}\mathcal{H}\nabla^{^1}_X \omega Y+\omega
\mathcal{T}_X \omega Y).\nonumber
\end{eqnarray}
Then, since $F$ is a linear isometry between $(ker F_*)^\perp$ and
$TM_2$, $(\nabla F_*)(X,Y)=0$ if and only if
$\omega\hat{\nabla}_X\phi Y+\mathcal{C}\mathcal{T}_X \phi
Y+\mathcal{C}\mathcal{H}\nabla^{^1}_X \omega Y+\omega
\mathcal{T}_X \omega Y=0$. Thus $(\nabla F_*)(X,Y)=0$ if and only
if
\begin{equation}
\omega(\hat{\nabla}_X\phi Y+ \mathcal{T}_X \omega Y)=0,\,
\mathcal{C}(\mathcal{T}_X \phi Y+\mathcal{H}\nabla^{^1}_X \omega
Y)=0. \label{eq:3.9}
\end{equation}
In a similar way, for $X \in \Gamma(ker F_*)$ and $Z \in \Gamma((ker
F_*)^\perp)$, $(\nabla F_*)(X,Z)=0$ if and only if
\begin{equation}
\omega(\hat{\nabla}_X\mathcal{B} Z+ \mathcal{T}_X
\mathcal{C}Z)=0,\, \mathcal{C}(\mathcal{T}_X
\mathcal{B}Z+\mathcal{H}\nabla^{^1}_X \mathcal{C}Z)=0.
\label{eq:3.10}
\end{equation}
Then proof follows from (\ref{eq:3.8})-(\ref{eq:3.10}).
\end{proof}

We now investigate the geometry of leaves of the  distribution $(ker
F_*)^\perp$.

\begin{prop}\label{Proposition 3.5}Let $F$ be a
semi-invariant submersion from a K\"{a}hler manifold $(M_1,g_1,J)$
onto a Riemannian manifold $(M_2,g_2)$. Then the distribution $(ker
F_*)^\perp$ defines a totally geodesic foliation if and only if
$$\mathcal{A}_{Z_1}\mathcal{B}Z_2+\mathcal{H}\nabla^{^1}_{Z_1}\mathcal{C}Z_2
\in
\Gamma(\mu),\,\mathcal{A}_{Z_1}\mathcal{C}Z_2+\mathcal{V}\nabla^{^1}_{Z_1}Z_2\in
\Gamma(\mathcal{D}_2)$$ for $Z_1,Z_2 \in \Gamma((ker F_*)^\perp)$.
\end{prop}

\begin{proof}From (\ref{eq:2.1}) and (\ref{eq:2.2})
we have $\nabla^{^1}_{Z_1}Z_2=-J\nabla^{^1}_{Z_1}JZ_2$ for $Z_1,Z_2
\in \Gamma((ker F_*)^\perp)$. Using (\ref{eq:3.4}), (\ref{eq:2.9})
and (\ref{eq:2.10}) we obtain
\begin{eqnarray}
\nabla^{^1}_{Z_1}Z_2&=&-J(A_{Z_1}\mathcal{B}Z_2+\mathcal{V}\nabla^{^1}_{Z_1}\mathcal{B}Z_2)\nonumber\\
               &-&J(\mathcal{H}\nabla^{^1}_{Z_1}\mathcal{C}Z_2+\mathcal{A}_{Z_1}\mathcal{C}Z_2).\nonumber
\end{eqnarray}
Then by using (\ref{eq:3.3}) and (\ref{eq:3.4}) we get
\begin{eqnarray}
\nabla^{^1}_{Z_1}Z_2&=&-\mathcal{B}A_{Z_1}\mathcal{B}Z_2-\mathcal{C}A_{Z_1}\mathcal{B}Z_2+\phi \mathcal{V}\nabla^{^1}_{Z_1}\mathcal{B}Z_2\nonumber\\
               &-&\omega \mathcal{V}\nabla^{^1}_{Z_1}\mathcal{B}Z_2-\mathcal{B}\mathcal{H}\nabla^{^1}_{Z_1}\mathcal{C}Z_2-\mathcal{C}\mathcal{H}\nabla^{^1}_{Z_1}\mathcal{C}Z_2\nonumber\\
               &-&\phi \mathcal{A}_{Z_1}\mathcal{C}Z_2-\omega\mathcal{A}_{Z_1}\mathcal{C}Z_2.\nonumber
\end{eqnarray}
Hence, we have $\nabla^{^1}_{Z_1}Z_2 \in \Gamma((ker F_*)^\perp)$
if and only if
$$-\mathcal{B}A_{Z_1}\mathcal{B}Z_2-\phi
\mathcal{V}\nabla^{^1}_{Z_1}\mathcal{B}Z_2-\mathcal{B}\mathcal{H}\nabla^{^1}_{Z_1}\mathcal{C}Z_2-\phi
\mathcal{A}_{Z_1}\mathcal{C}Z_2=0.$$ Thus $\nabla^{^1}_{Z_1}Z_2
\in \Gamma((ker F_*)^\perp)$ if and only if
$$\mathcal{B}(A_{Z_1}\mathcal{B}Z_2+\mathcal{H}\nabla^{^1}_{Z_1}\mathcal{C}Z_2)=0,\,\phi
(\mathcal{V}\nabla^{^1}_{Z_1}\mathcal{B}Z_2+\mathcal{A}_{Z_1}\mathcal{C}Z_2)=0$$
which completes proof.
\end{proof}

In a similar way, we have the following result.

\begin{prop}\label{Proposition 3.6}Let $F$ be a
semi-invariant Riemannian submersion from a K\"{a}hler manifold
$(M_1,g_1,J)$ onto a Riemannian manifold $(M_2,g_2)$. Then the
distribution $ker F_*$ defines a totally geodesic foliation if and
only if
$$\mathcal{T}_{X_1} \phi X_2+\mathcal{H}\nabla^{^1}_{X_1}\omega X_2 \in
\Gamma(J\mathcal{D}_2),\hat{\nabla}_{X_1}\phi
X_2+\mathcal{T}_{X_1}\omega X_2 \in \Gamma(\mathcal{D}_1)$$ for
$X_1, X_2 \in \Gamma(ker F_*)$.
\end{prop}

From Proposition \ref{Proposition 3.6}, we have the following
result.

\begin{cor}\label{Corollary 3.7}Let $F$ be a semi-invariant
Riemannian submersion from a K\"{a}hler manifold $(M_1,g_1,J)$ onto
a Riemannian manifold $(M_2,g_2)$. Then the distribution $ker F_*$
defines a totally geodesic foliation if and only if
\begin{eqnarray}
g_2((\nabla F_*)(X_1,X_2),F_*(JZ))&=&0,\nonumber\\
g_2((\nabla F_*)(X_1,\omega
X_2),F_*(W))&=&-g_1(\mathcal{T}_{X_1}W,\phi X_2)\nonumber
\end{eqnarray}
for $X_1,X_2 \in \Gamma(ker F_*)$, $Z \in \Gamma(\mathcal{D}_2)$.
and $W \in \Gamma(\mu)$.
\end{cor}

\begin{proof}For $X_1, X_2 \in\Gamma(ker F_*)$,
$\hat{\nabla}_{X_1}\phi X_2+\mathcal{T}_{X_1}\omega X_2 \in
\Gamma(\mathcal{D}_1)$ if and only if $g_1(\hat{\nabla}_{X_1}\phi
X_2+\mathcal{T}_{X_1}\omega X_2,Z)=0$ for $Z \in
\Gamma(\mathcal{D}_2)$. Skew-symmetric $\mathcal{T}$ and
(\ref{eq:2.7}) imply that
\begin{eqnarray}
g_1(\hat{\nabla}_{X_1}\phi X_2+\mathcal{T}_{X_1}\omega
X_2,Z)&=&g_1(\nabla^{^1}_{X_1}\phi
X_2,Z)\nonumber\\
&-&g_1(\omega X_2,\mathcal{T}_{X_1}Z).\nonumber
\end{eqnarray}
Hence we have
\begin{eqnarray}
g_1(\hat{\nabla}_{X_1}\phi X_2+\mathcal{T}_{X_1}\omega
X_2,Z)&=&-g_1(\phi
X_2,\nabla^{^1}_{X_1}Z)\nonumber\\
&-&g_1(\omega X_2,\mathcal{T}_{X_1}Z).\nonumber
\end{eqnarray}
Using again (\ref{eq:2.7}) we get
\begin{eqnarray}
g_1(\hat{\nabla}_{X_1}\phi X_2+\mathcal{T}_{X_1}\omega
X_2,Z)&=&-g_1(JX_2,\hat{\nabla}_{X_1}Z)\nonumber\\
&-&g_1(\omega X_2,\mathcal{T}_{X_1}Z).\nonumber
\end{eqnarray}
Hence we have
$$
g_1(\hat{\nabla}_{X_1}\phi X_2+\mathcal{T}_{X_1}\omega
X_2,Z)=-g_1(JX_2,\nabla^{^1}_{X_1}Z).$$ Then from (\ref{eq:2.2})
we derive
$$
g_1(\hat{\nabla}_{X_1}\phi X_2+\mathcal{T}_{X_1}\omega
X_2,Z)=g_1(X_2,\nabla^{^1}_{X_1}JZ).$$ Thus we have
$$
g_1(\hat{\nabla}_{X_1}\phi X_2+\mathcal{T}_{X_1}\omega
X_2,Z)=-g_1(\nabla^{^1}_{X_1}X_2,JZ).$$ Then Riemannian submersion
$F$ implies that
$$
g_1(\hat{\nabla}_{X_1}\phi X_2+\mathcal{T}_{X_1}\omega
X_2,Z)=-g_2(F_*(\nabla^{^1}_{X_1}X_2),F_*(JZ)).$$ Using
(\ref{eq:2.11}) we get
\begin{equation}
g_1(\hat{\nabla}_{X_1}\phi X_2+\mathcal{T}_{X_1}\omega
X_2,Z)=g_2((\nabla F_*)(X_1,X_2),F_*(JZ)).\label{eq:3.11}
\end{equation}
On the other hand, for $X_1, X_2 \in \Gamma(ker F_*)$,
$\mathcal{T}_{X_1} \phi X_2+\mathcal{H}\nabla^{^1}_{X_1}\omega X_2
\in \Gamma(J\mathcal{D}_2)$ if and only if $g_1(\mathcal{T}_{X_1}
\phi X_2+\mathcal{H}\nabla^{^1}_{X_1}\omega X_2, W)=0$ for $W\in
\Gamma(\mu)$. Since $\mathcal{T}$ is skew-symmetric, we have
\begin{eqnarray}
g_1(\mathcal{T}_{X_1} \phi X_2+\mathcal{H}\nabla^{^1}_{X_1}\omega
X_2, W)&=&-g_1( \phi X_2,
\mathcal{T}_{X_1}W)\nonumber\\
&+&g_1(\nabla^{^1}_{X_1}\omega X_2, W).\nonumber
\end{eqnarray}
Since $F$ is a Riemannian submersion, we get
\begin{eqnarray}
g_1(\mathcal{T}_{X_1} \phi X_2+\mathcal{H}\nabla^{^1}_{X_1}\omega
X_2, W)&=&-g_1( \phi X_2,
\mathcal{T}_{X_1}W)\nonumber\\
&+&g_2(F_*(\nabla^{^1}_{X_1}\omega X_2),F_*W).\nonumber
\end{eqnarray}
Then from (\ref{eq:2.11}) we arrive at
\begin{eqnarray}
g_1(\mathcal{T}_{X_1} \phi X_2+\mathcal{H}\nabla^{^1}_{X_1}\omega
X_2, W)&=&-g_1( \phi X_2,
\mathcal{T}_{X_1}W)\nonumber\\
&+&g_2(-(\nabla F_*)(X_1,\omega X_2),F_*W).\label{eq:3.12}
\end{eqnarray}
Thus proof follows from (\ref{eq:3.11}), (\ref{eq:3.12}) and
Proposition \ref{Proposition 3.6}
\end{proof}

From Proposition \ref{Proposition 3.3} and Proposition
\ref{Proposition 3.5} we have the following.

\begin{thrm}\label{Theorem 3.8}Let $F$ be a semi-invariant
submersion from a K\"{a}hler manifold $(M_1,g_1,J)$ onto a
Riemannian manifold $(M_2,g_2)$. Then $M_1$ is locally a product
Riemannian manifold $M_{\mathcal{D}_1}\times M_{\mathcal{D}_2}\times
M_{(ker F_*)^\perp}$ if and only if
$$(\nabla \phi)=0\quad \mathrm{on }\quad  ker F_*$$
and
$$\mathcal{A}_{Z_1}\mathcal{B}Z_2+\mathcal{H}\nabla^{^1}_{Z_1}\mathcal{C}Z_2
\in
\Gamma(\mu),\mathcal{A}_{Z_1}\mathcal{C}Z_2+\mathcal{V}\nabla^{^1}_{Z_1}Z_2\in
\Gamma(\mathcal{D}_2)$$ {\it for} $Z_1,Z_2 \in \Gamma((ker
F_*)^\perp)$, where $M_{\mathcal{D}_1}$, $M_{\mathcal{D}_2}$ and
$M_{(ker F_*)^\perp}$ are integral manifolds of the distributions
$\mathcal{D}_1$, $\mathcal{D}_2$ and $(ker F_*)^\perp$.
\end{thrm}

Also from Corollary \ref{Corollary 3.7} and Proposition
\ref{Proposition 3.5}, we have the following result.

\begin{thrm}\label{Theorem 3.9}Let $F$ be a semi-invariant
submersion from a K\"{a}hler manifold $(M_1,g_1,J)$ onto a
Riemannian manifold $(M_2,g_2)$. Then $M_1$ is locally a product
Riemannian manifold $M_{ker F_*}\times M_{(ker F_*)^\perp}$ if and
only if
\begin{eqnarray}
g_2((\nabla F_*)(X_1,X_2),F_*(JZ))&=&0,\nonumber\\
g_2((\nabla F_*)(X_1,\omega
X_2),F_*(W))&=&-g_1(\mathcal{T}_{X_1}W,\phi X_2)\nonumber
\end{eqnarray}
and
$$\mathcal{A}_{Z_1}\mathcal{B}Z_2+\mathcal{H}\nabla^{^1}_{Z_1}\mathcal{C}Z_2
\in
\Gamma(\mu),\mathcal{A}_{Z_1}\mathcal{C}Z_2+\mathcal{V}\nabla^{^1}_{Z_1}Z_2\in
\Gamma(\mathcal{D}_2)$$ for $X_1,X_2 \in \Gamma(ker F_*)$, $W \in
\Gamma(\mu)$, $Z \in \Gamma(\mathcal{D}_2)$ and $Z_1,Z_2 \in
\Gamma((ker F_*)^\perp)$, where $M_{ker F_*}$ and
$M_{(ker F_*)^\perp}$ are integral manifolds of the distributions $ker F_*$ and $(ker F_*)^\perp$.
\end{thrm}

\section{Semi-invariant submersions with totally umbilical fibers}\label{sct4}

In this section we give two theorems on semi-invariant submersions
with totally umbilical fibers. First result shows that a
semi-invariant submersion puts some restrictions on total
manifolds. Also we obtain a classification for such submersions.
Let $F$ be a Riemannian submersion from a Riemannian manifold onto
a Riemannian manifold $(M_2,g_2)$. Recall that a Riemannian
submersion is called a Riemannian submersion with totally
umbilical fibers if
\begin{equation}
\mathcal{T}_X Y=g_1(X,Y)H \label{eq:4.1}
\end{equation}
for $X, Y \in \Gamma(ker F_*)$, where $H$ is the mean curvature
vector field of the fiber. We also recall that a simply connected
complete K\"ahler manifold of constant sectional curvature $c$ is
called a complex space-form, denoted by $M(c)$. The curvature
tensor of $M(c)$ is
\begin{eqnarray}
R(X,\,Y)Z & = & \frac{c}{4}\,[ g\,(Y,\,Z)X - g\,(X,\,Z)Y + g\,(JY,\,Z)JX  \nonumber \\
          & - & g\,(JX,\,Z)JY + 2\,g\,(X,\,JY)JZ ] \label{eq:4.2} \end{eqnarray}
for $X,Y,Z \in \Gamma(TM)$. Moreover, from \cite{O'Neill} we have
the following relation for a Riemannian submersion
\begin{equation}
g_1(R^1(X_1,X_2)X_3,Z)=g_1((\nabla_{X_2} \mathcal{T})_{X_1}
X_3,Z)-g_1((\nabla_{X_1} \mathcal{T})_{X_2} X_3,Z) \label{eq:4.3}
\end{equation}
 for $X_1,X_2,X_3 \in \Gamma(ker F_*)$ and $Z
\in \Gamma((ker F_*)^\perp)$, where $R^1$ is the curvature tensor
field of $M_1$
and $(\nabla \mathcal{T})$ is the covariant derivative of $\mathcal{T}$.\\

By using (\ref{eq:4.1}), (\ref{eq:4.2}) and (\ref{eq:4.3}), as in
CR-submanifolds, see: Theorem 1.2 of \cite[p.78]{Bejancu}, we have
the following result.

\begin{thrm}\label{Theorem 4.1}Let $F$ be a semi-invariant
submersion with totally umbilical fibers from a complex space form
$(M_1(c),g_1,J)$ onto a Riemannian manifold $(M_2,g_2)$. Then $c=0$.
\end{thrm}

We now give a classification theorem for semi-invariant Riemannian
submersions with totally umbilical fibers. But we need the following
result which shows that the mean curvature vector field of
semi-invariant Riemannian submersions has special form.

\begin{lem}\label{Lemma 4.2}Let $F$ be a semi-invariant
submersion with totally umbilical fibers from a K\"{a}hler manifold
$(M_1,g_1,J)$ onto a Riemannian manifold $(M_2,g_2)$. Then $H \in
\Gamma(J\mathcal{D}_2)$.
\end{lem}

\begin{proof}Using (\ref{eq:2.1}), (\ref{eq:2.2}),
(\ref{eq:2.7}), (\ref{eq:3.3}) and (\ref{eq:3.4}) we get
$$\mathcal{T}_{X_1}JX_2+\hat{\nabla}_{X_1}JX_2=\mathcal{B}\mathcal{T}_{X_1}X_2+\mathcal{C}\mathcal{T}_{X_1}X_2+\phi
\hat{\nabla}_{X_1}X_2+\omega \hat{\nabla}_{X_1}X_2$$ for $X_1,X_2
\in \Gamma(\mathcal{D}_1)$.  Thus, for $W \in \Gamma(\mu)$, we
obtain
$$g_1(\mathcal{T}_{X_1}JX_2,W)=g_1(\mathcal{C}\mathcal{T}_{X_1}X_2,W).$$
Using (\ref{eq:4.1}) we derive
$$g_1(X_1,JX_2)g_1(H,W)=g_1(J\mathcal{T}_{X_1}X_2,W).$$
Hence we have
$$g_1(X_1,JX_2)g_1(H,W)=-g_1(\mathcal{T}_{X_1}X_2,JW).$$
Using again (\ref{eq:4.1}) we arrive at
\begin{equation}
g_1(X_1,JX_2)g_1(H,W)=-g_1(X_1,X_2)g_1(H,JW).\label{eq:4.6}
\end{equation} Interchanging the role of $X_1$ and $X_2$, we obtain
\begin{equation}
g_1(X_2,JX_1)g_1(H,W)=-g_1(X_2,X_1)g_1(H,JW).\label{eq:4.7}
\end{equation}
Thus from (\ref{eq:4.6}) and (\ref{eq:4.7}) we derive
$$g_1(X_1,X_2)g_1(H,JW)=0$$
which shows that $H \in \Gamma(J\mathcal{D}_2)$ due to $\mu$ is
invariant distribution.
\end{proof}

We now give a classification theorem for a semi-invariant submersion
with totally umbilical fibers which is similar to that Theorem 6.1
of \cite[p.96]{Yano-Kon2}, therefore we omit its proof. We note that
Lemma \ref{Lemma 4.2} implies that one can use the method which was
used in the proof of Theorem 6.1 of \cite{Yano-Kon2}.

\begin{thrm}\label{Theorem 4.3}Let $F$ be a semi-invariant
submersion with totally umbilical fibers from a K\"{a}hler manifold
$(M_1,g_1,J)$ onto a Riemannian manifold $(M_2,g_2)$. Then either
$\mathcal{D}_2$ is one dimensional or the fibers are totally
geodesic.
\end{thrm}

\end{document}